\documentclass[10pt,twocolumn,twoside]{IEEEtran} 

\IEEEoverridecommandlockouts

\usepackage{times,amsmath,epsfig,amssymb,graphicx,amsfonts,amsthm}
\usepackage{empheq}
\usepackage{graphicx}
\usepackage{psfrag}
\usepackage{fge}
\usepackage{caption}
\usepackage{subcaption}
\usepackage{amsmath}
\usepackage{amsthm}
\usepackage{color}
\usepackage{amsfonts}   
\usepackage{amssymb}    
\usepackage{mathrsfs}
\usepackage{setspace}
\usepackage{dblfloatfix}
\usepackage{tikz}
\usepackage{tkz-tab}
\usepackage{algorithm}
\usepackage{algcompatible}
\usepackage{lipsum}
\usepackage{color}
\usepackage{mathrsfs}
\usepackage{epstopdf}
\usepackage{enumerate}
\usepackage{cite}
\floatname{algorithm}{Algorithm}
\newtheorem{myfact}{Fact}
\newtheorem{mydef}{Definition}
\newtheorem{myassump}{Assumption}
\newtheorem{mytheorem}{Theorem}

\newtheorem{mycorollary}{Corollary}
\newtheorem{myproposition}{Proposition}

\newcounter{ale}

\newenvironment{liste}{\begin{itemize}}{\end{itemize}}
\newcommand{\aliste}{\begin{liste} \setcounter{ale}{1}}
\newcommand{\zliste}{\end{liste}}

\title{{\LARGE {\bf Minimal Actuator Placement with Bounds on Control Effort}}}

\author{V.~Tzoumas, M.~A.~Rahimian, G.~J.~Pappas, A.~Jadbabaie{$^{\star}$}
\thanks{$^{\star}$All authors are with the Department of Electrical and Systems Engineering, University of Pennsylvania, Philadelphia, PA 19104-6228 USA (email: {\fontsize{8}{8}\selectfont\ttfamily\upshape \{vtzoumas, mohar, pappasg, jadbabai\}@seas.upenn.edu}).}

\thanks{This work was supported in part by ARO MURI W911NF-12-1-0509, in part TerraSwarm, one of six centers of STARnet, a Semiconductor Research Corporation program sponsored by MARCO and DARPA, and in part by AFOSR Complex Networks Program.}

}

\begin{document}
\maketitle

\begin{abstract} We address the problem of minimal actuator placement in linear systems so that the volume of the set of states reachable with one unit or less of input energy is lower bounded by a desired value.  First, following the recent work of Olshevsky, we prove that this is NP-hard.  Then,  we provide an efficient algorithm which, {{for a given range of problem parameters}}, approximates up to a multiplicative factor of $O(\log n)$,  $n$ being the network size, any optimal actuator set that meets the same energy criteria; this is the best approximation factor one can achieve in polynomial time, in the worst case. Moreover, the algorithm uses a perturbed version of the involved control energy metric, which we prove to be supermodular.  Next, we focus on the related problem of cardinality-constrained actuator placement for minimum control effort, where the optimal actuator set is selected to maximize the volume of the set of states reachable with one unit or less of input energy.  While this is also an NP-hard problem, we use our proposed algorithm to efficiently approximate its solutions as well.
\end{abstract}


\section{Introduction}\label{sec:Intro}
During the past decade, an increased interest in the analysis of large-scale systems has led to a variety of studies that range from the mapping of the human's brain functional connectivity to the understanding of the collective behavior of animals, and the evolutionary mechanisms of complex ecological systems~\cite{newman2006structure,hermundstad2013structural,katz2011inferring,jordan2004network}.  At the same time, control scientists develop methods for the regulation of such complex systems, with the notable examples in~\cite{orosz2010controlling}, for the control of biological systems;~\cite{rajapakse2012can}, for the regulation of brain and neural networks;~\cite{khanafer2014information}, for robust information spread over social networks, and~\cite{Chen-2012-DR-Springer}, for load management in smart grid.

On the other hand, the large size of these systems, as well as the need for low cost control, has made the identification of a small fraction of their states, to steer them around the entire space, an important problem~\cite{2013arXiv1304.3071O, ramos2014np,summers2016submodularity, bullo2014}.  This is a  task of formidable complexity; indeed, it is shown in~\cite{2013arXiv1304.3071O} that finding a small number of actuators, so that a linear system is controllable, is NP-hard.  However, mere controllability is of little value if the required input energy for the desired transfers is exceedingly high, when, for example, the controllability matrix is close to singularity~\cite{Chen:1998:LST:521603}.  Therefore, by choosing input states to ensure controllability alone, one may not achieve a cost-effective control for the system. 

{In this paper, we address this important requirement by providing efficient approximation algorithms to actuate a small fraction of a system's states so that a specified control energy performance over the entire state space is guaranteed.} In particular, we first consider the selection of a minimal number of actuated states so that a pre-specified lower bound on the volume of the set of states reachable with one or less units of input energy is satisfied.  Finding such a subset of states is a challenging task, since it involves the search for a small number of actuators that induce controllability, which constitutes a combinatorial problem that can be computationally intensive. Indeed, identifying a small number of actuated states for inducing controllability alone is NP-hard~\cite{2013arXiv1304.3071O}.  Therefore, we extend this computationally hard problem by introducing an energy performance requirement on the choice of the optimal actuator set, and we solve it with an efficient approximation algorithm.

Specifically, we first generalize the involved energy objective to an $\epsilon$-close one, which remains well-defined even for actuator sets that render the system uncontrollable.  Then, we make use of this metric and relax the implicit controllability constraint from the original actuator placement problem.  Notwithstanding, we prove that for small values of $\epsilon$ all solutions of this auxiliary program still render the system controllable.  This fact, along with the supermodularity of the generalized objective with respect to the choice of the actuator set, leads to an efficient algorithm which, {for a given range of problem parameters}, approximates up to a multiplicative factor of $O(\log n)$, where $n$ is the size of the system, any optimal actuator set that meets the specified energy criterion.  Moreover, this is the best approximation factor one can achieve in polynomial time, in the worst case.  Hence, with this algorithm {we address the open problem of minimal actuator placement subject to bounds on the control effort}~\cite{2013arXiv1304.3071O, summers2016submodularity,bullo2014, PhysRevLett.108.218703,PhysRevLett.110.208701}.  

Relevant results are also found in~\cite{bullo2014}, where the authors study the controllability of a system with respect to the smallest eigenvalue of the controllability Gramian, and they derive a lower bound on the number of actuators so that this eigenvalue is lower bounded by a fixed value.  Nonetheless, they do not provide an algorithm to identify the actuators that achieve this value.  

Next, we consider the problem of cardinality-constrained actuator placement for minimum control effort, where the optimal actuator set is selected so that the volume of the set of states that
can be reached with one unit or less of input energy is maximized.  The most related works to this
problem are the \cite{summers2016submodularity} and \cite{7039832}, in which the authors assume a controllable system and consider the problem of choosing a few extra actuators in order to optimize some of the input energy metrics proposed in~\cite{Muller1972237}.  Their main contribution is in observing that these energy metrics are supermodular with respect to the choice of the extra actuated states.  The assumption of a controllable system is necessary since these metrics depend on the inverse of the controllability Gramian, {as they capture the control energy for steering the system around the entire state space}.  Nonetheless, it should be also clear that making a system controllable by first placing some actuators to ensure controllability alone, and then adding some extra ones to optimize a desired energy metric, introduces a sub-optimality that is carried over to the end result.  In this paper, we follow a parallel line of work to the minimal actuator placement problem, and provide an efficient algorithm that selects {all the actuated states to maximize the volume of the set of states that
can be reached with one unit or less of input energy without any assumptions on the controllability of the involved system}.  

A similar actuator placement problem is studied in~\cite{bullo2014} for stable systems.  Nevertheless, its authors propose a heuristic actuator placement procedure that does not constrain the number of available actuators and {does not optimize} their control energy objective.
Our proposed algorithm selects {a cardinality-constrained actuator set that minimizes a control energy metric, even for unstable systems.}

The remainder of this paper is organized as follows.  The formulation and model for the actuator placement problems are set forth in Section~\ref{sec:Prelim}, where the corresponding integer optimization programs are stated.  In Sections~\ref{sec:Min_N} and~\ref{sec:Min_E} we discuss our main results, including the intractability of these problems, as well as the supermodularity of the involved control energy metrics with respect to the choice of the actuator sets.  Then, we provide efficient approximation algorithms for their solution that guarantee a specified control energy performance over the entire state space.  Section~\ref{sec:conc} concludes the paper.

\section{Problem Formulation} \label{sec:Prelim}

\paragraph*{Notation}

We denote the set of natural numbers $\{1,2,\ldots\}$ as $\mathbb{N}$, the set of real numbers as  $\mathbb{R}$,  and we let $[n]\equiv \{1, 2, \ldots, n\}$ for all $n \in \mathbb{N}$.  Also, given a set $\mathcal{X}$, we denote as $|\mathcal{X}|$ its cardinality.  Matrices are represented by capital letters and vectors by lower-case letters. For a matrix ${A}$, ${A}^{T}$ is its transpose and $A_{ij}$ is its element located at the $i-$th row and $j-$th column.  If ${A}$ is positive semi-definite or positive definite, we write ${A} \succeq {0}$ and ${A}\succ {0}$, respectively. Moreover, for $i \in [n]$, we let ${I}^{(i)}$ be an $n \times n$ matrix with a single non-zero element: $I_{ii}=1$, while $I_{jk}=0$, for $j$, $k\neq i$.  Furthermore, we denote as ${I}$ the identity matrix, whose dimension is inferred from the context.  Additionally, for ${\delta} \in \mathbb{R}^n$, we let $\text{diag}({\delta})$ denote an $n \times n$ diagonal matrix such that $\text{diag}({\delta})_{ii}=\delta_i$ for all $i \in [n]$.  {Finally, we set $\{0,1\}^n$ to be the set of vectors in $\mathbb{R}^n$ whose elements are either zero or one.}

\subsection{Actuator Placement Model}

Consider a linear system of $n$ states, $x_1, x_2,\ldots,x_n$, whose evolution is described by
\begin{align}
\dot{{x}}(t) = {A}{{x}}(t) + {B}{{u}}(t), t > t_0,
\label{eq:dynamics}
\end{align} where  $t_0 \in \mathbb{R}$ is fixed, ${x}\equiv \{x_1,x_2,\ldots,x_n\}$, $\dot{{x}}(t)\equiv d{x}/dt$, while ${u}$ is the corresponding input vector.  The matrices ${A}$ and ${B}$ are of appropriate dimension.  We equivalently refer to~\eqref{eq:dynamics} as a network of $n$ nodes, $1, 2,\ldots, n$, which we associate with the states $x_1, x_2,\ldots, x_n$, respectively.  Moreover, we denote their collection as $\mathcal{V}\equiv[n]$.

Henceforth, ${A}$ is given while B is a \textit{diagonal zero-one} matrix that we design
so that~\eqref{eq:dynamics} satisfies a specified control energy criterion over the entire state space.

\begin{myassump}\label{assump:Diag_B}
${B}=\emph{\text{diag}}({\delta})$, where ${\delta}\in\{0,1\}^{n}$.
\end{myassump}

Specifically, if $\delta_i=1$, state $x_i$ may receive an input, while if $\delta_i=0$, it receives none. 

\begin{mydef}[{Actuator Set, Actuator}]
Given a $\delta \in \{0,1\}^{n}$, let $\Delta\equiv\{i: i\in \mathcal{V} \text{ and } \delta_i=1\}$; then, $\Delta$ is called an \emph{actuator set} and each $i \in \Delta$ an \emph{actuator}.
\end{mydef}

\subsection{Controllability and Related  Energy Metrics}\label{subsec:cntr_en}

We consider the notion of controllability and relate it to the problems of this paper, i.e., the minimal actuator placement for constrained control energy and the cardinality-constrained actuator placement for minimum control effort.

System~\eqref{eq:dynamics} is controllable --- equivalently, $(A,B)$ is controllable --- if for any finite $t_1>t_0$ and any initial state ${x}_0\equiv {x}(t_0)$ it can be steered to any other state ${x}_1\equiv {x}(t_1)$ by some input ${u}(t)$ defined over $[t_0, t_1]$.  Moreover, for general matrices ${A}$ and ${B}$, the controllability condition is equivalent to the matrix 
\begin{align}
{W}\equiv \int_{t_0}^{t_1} \mathrm{e}^{{A}(t-t_0)} {B}{B}^{T} \mathrm{e}^{{A}^{T}(t-t_0)}\,\mathrm{d}{t},\label{eq:general_gramian}
\end{align} being positive definite for any $t_1>t_0$~\cite{Chen:1998:LST:521603}.  Therefore, we refer to ${W}$ as the \textit{controllability matrix} of~\eqref{eq:dynamics}.

The controllability of a linear system is of interest because it is related to the solution of the following minimum-energy transfer problem

\begin{equation}\label{pr:min_energy_transfer}
\begin{aligned}
 \underset{{u}(\cdot)}{\text{minimize}} & \; \;  \; 
 \int_{t_0}^{t_1} {u}(t)^T{u}(t)\,\mathrm{d}{t}\\
\text{subject to} \\
& \dot{{x}}(t) = {A}{{x}}(t) + {B}{{u}}(t), t_0 <t \leq t_1,\\
& {x}(t_0)= 0, {x}(t_1)={x_1},
\end{aligned}
\end{equation} where ${A}$ and ${B}$ are any matrices of appropriate dimension.

In particular,  if for the given ${A}$ and ${B}$~\eqref{eq:dynamics} is controllable the resulting minimum control energy is given by
\begin{align}
{x}_1^{T}{W}^{-1}{x}_1, \label{eq:min_energy}
\end{align}
where $\tau=t_1-t_0$~\cite{Muller1972237}. 
Thereby, the states that belong to the eigenspace of the smallest eigenvalues of~\eqref{eq:general_gramian} require higher energies of control input~\cite{Chen:1998:LST:521603}.  Extending this observation along all the directions of transfers in the state space, we infer that the closer ${W}$ is to singularity the larger the expected input energy required for these transfers to be achieved~\cite{Muller1972237}.  For example, consider the case where ${W}$ is singular, i.e., when there exists at least one direction along which system~\eqref{eq:dynamics} cannot be steered~\cite{Chen:1998:LST:521603}.  Then, the corresponding minimum control energy along this direction is \textit{infinity}.  

This motivates the consideration of control energy metrics that quantify the {{steering energy} along all the directions in the state space}, as the $\log\det({W}^{-1})$~\cite{Muller1972237}.  Indeed, this metric is {well-defined only for controllable systems} --- $W$ must be invertible --- and is {directly related} to~\eqref{eq:min_energy}.  In more detail, {$\sqrt{\det({W}^{-1})}$ is inversely proportional to the volume of the set of states reachable with one or less units of input energy, i.e., the volume of $\{x: {x}^{T}{W}^{-1}{x}\leq 1\}$; as a result, when $\log\det({W}^{-1})$ is minimized, the volume of $\{x: {x}^{T}{W}^{-1}{x}\leq 1\}$ is maximized.}     In this paper, we aim to select a small number of actuators for system~\eqref{eq:dynamics} so that $\log\det({W}^{-1})$ either meets a specified upper bound or is minimized.

Per Assumption~\ref{assump:Diag_B}, further properties for the controllability matrix are due: For any actuator set $\Delta$, let ${W}_\Delta \equiv {W}$; then,
\begin{align}
{W}_{\Delta} = \sum_{i=1}^n \delta_i {W}_i,
\label{eq:gramianTOdelta}
\end{align}where ${W}_i \equiv \int_{t_0}^{t_1} \mathrm{e}^{{A}t} {I}^{(i)} \mathrm{e}^{{A}^{T} t}\,\mathrm{d}{t}$ for any $i \in [n]$.  This follows from~\eqref{eq:general_gramian} and the fact that ${B}{B}^{T}={B}=\sum_{i = 1}^{n} \delta_i {I}^{(i)}$ for ${B} =\text{diag}({\delta})$.  Finally, for any ${\Delta_1}\subseteq{\Delta_2}\subseteq \mathcal{V}$,~\eqref{eq:gramianTOdelta} and ${W}_1, W_2, \ldots, W_n \succeq {0}$ imply ${W}_{\Delta_1}\preceq {W}_{\Delta_2}$. 

\subsection{Actuator Placement Problems}\label{subsec:problems}

We consider the selection of a small number of actuators for system~\eqref{eq:dynamics} so that $\log\det({W}^{-1})$ either satisfies an upper bound or is minimized.  The challenge is in doing so with as few actuators as possible.  This is an important improvement over the existing literature where the goal of actuator placement problems has either been to ensure controllability alone~\cite{2013arXiv1304.3071O} or the weaker property of structural controllability~\cite{jafari2011leader,Commault20133322}.  Other relevant results consider the task of leader-selection~\cite{clark2014_2,clark2014_1}, where the leaders are the actuated states and are chosen so to minimize a mean-square convergence error of the remaining states.  

Furthermore, the most relevant works to our study are the \cite{summers2016submodularity} and \cite{7039832} since its authors consider the minimization of $\log\det({W}^{-1})$; nevertheless, their results rely on a pre-existing actuator set that renders~\eqref{eq:dynamics} controllable although this set is not selected for the minimization of this energy metric.   One of our contributions is in achieving optimal actuator placement for minimum control effort {without assuming controllability beforehand.} Also, the authors of~\cite{bullo2014} adopt a similar framework for actuator placement but focus on deriving an upper bound for the smallest eigenvalue of $W$ with respect to the number of actuators and a lower bound for the required number actuators so that this eigenvalue takes a specified value.  In addition, they consider the maximization of $\text{tr}({W})$;
however, their techniques cannot be applied when minimizing the $\log\det({W}^{-1})$, while the maximization of $\text{tr}({W})$ may not ensure controllability~\cite{bullo2014}.  

We next provide the exact statements of our actuator placement problems, while their solution analysis follows in Sections~\ref{sec:Min_N} and~\ref{sec:Min_E}.  We first consider the problem
\begin{equation}\tag{I}\label{pr:min_set}
\begin{aligned}
 \underset{\Delta \subseteq \mathcal{V}}{\text{minimize}} & \; \;  \; 
 |\Delta|\\
\text{subject to} \\
&\log\det({W}_\Delta^{-1}) \leq E,
\end{aligned}
\end{equation} 
for some constant $E$. Its domain is $\{\Delta:  \Delta \subseteq \mathcal{V} \text{ and } (A,B(\Delta)) \text{ is controllable}\}$ since the controllability matrix $W_{(\cdot)}$ must be invertible.  {Moreover, it is NP-hard, as we prove in Appendix \ref{appen:NP}.}

Additionally, Problem~\eqref{pr:min_set} is feasible for certain values of $E$.  In particular, for any $\Delta$ such that $(A,B(\Delta))$ is controllable, $0 \prec {W}_\Delta$, i.e., $\log\det({W}_\mathcal{V}^{-1})\leq \log\det({W}_\Delta^{-1})$  since for any $\Delta$~\eqref{eq:gramianTOdelta} implies ${W}_\Delta\preceq {W}_\mathcal{V}$~\cite{bernstein2009matrix}; thus,~\eqref{pr:min_set} {is feasible for}
\begin{align}
E\geq \log\det({W}_\mathcal{V}^{-1}).\label{lo_bound_E}
\end{align} 

Moreover,~\eqref{pr:min_set} is a generalized version of the minimal controllability problem of~\cite{2013arXiv1304.3071O} so that its solution not only ensures controllability but also satisfies a guarantee in terms of a control energy metric; indeed, for $E\rightarrow\infty$ we recover the problem of~\cite{2013arXiv1304.3071O}.

We next consider the problem

\begin{equation}\tag{II}\label{pr:original}
\begin{aligned}
 \underset{\Delta \subseteq \mathcal{V}}{\text{minimize}} & \; \;  \; \log\det({W}_\Delta^{-1}) \\
\text{subject to} \\
&  |\Delta|\leq r,
\end{aligned}
\end{equation} where the goal is to find at most $r$ actuated states so that the volume of the set of states that
can be reached with one unit or less of input energy is maximized.  Its domain is $\{\Delta : \Delta\subseteq \mathcal{V}, |\Delta|\leq r \text{ and } (A,B(\Delta)) \text{ is controllable}\}$.  {Moreover, due to the NP-hardness of Problem \eqref{pr:min_set}, Problem \eqref{pr:original} is also NP-hard (cf.~Appendix \ref{appen:NP}).}

{Because~\eqref{pr:min_set} and~\eqref{pr:original} are NP-hard, we need to identify efficient approximation algorithms for their general solution}; this is the subject of Sections~\ref{sec:Min_N} and~\ref{sec:Min_E}.  In particular, in Section~\ref{sec:Min_N} we consider Problem~\eqref{pr:min_set} and provide for it a best approximation algorithm, {for a given range of problem parameters}.  {To this end, we first define an auxiliary program, which ignores the controllability constraint of~\eqref{pr:min_set}, and, nevertheless, admits an efficient approximation algorithm whose solutions not only satisfy an energy bound that is $\epsilon$-close to the original one but also render system~\eqref{eq:dynamics} controllable.}  Then, in Section~\ref{sec:Min_E} we turn our attention to~\eqref{pr:original}, and following a parallel line of thought as for~\eqref{pr:min_set}, we efficiently solve this problem as well. 

{Since the approximation algorithm for the aforementioned auxiliary program for \eqref{pr:min_set} is based on results for supermodular functions, we present below a brief overview of the relevant concepts.  The reader may consult \cite{krause2012submodular} for a survey on these results.}

\subsection{Supermodular Functions}\label{subsec:supermod}

{
We give the definition of a supermodular function, as well as, a relevant result that will be used in Section \ref{sec:Min_N} to construct an approximation algorithm for Problem \eqref{pr:min_set}.  The material of this section is drawn from \cite{citeulike:416650}.}

Let $\mathcal{V}$ be a finite set and denote as $2^\mathcal{V}$ its power set.

\begin{mydef}[Submodularity and supermodularity]
{A function $h:2^\mathcal{V}\mapsto \mathbb{R}$ is \emph{submodular} if for any sets $\Delta$ and $\Delta'$, with $\Delta \subseteq \Delta' \subseteq \mathcal{V}$, and any $a\notin \Delta'$,
\[
h(\Delta \cup \{a\})-h(\Delta)\geq h(\Delta' \cup \{a\})-h(\Delta').
\]
A function $h:2^\mathcal{V}\mapsto \mathbb{R}$ is \emph{supermodular} if $(-h)$ is submodular.}
\end{mydef}

{An alternative definition of a submodular function is based on the notion of non-increasing set functions.}

\begin{mydef}[Non-increasing and non-decreasing Set Function]
{A function $h:2^\mathcal{V}\mapsto \mathbb{R}$ is a \emph{non-increasing set function} if for any $\Delta \subseteq \Delta' \subseteq \mathcal{V}$, $h(\Delta)\geq h(\Delta')$. Moreover, $h$ is a \emph{non-decreasing set function} if $(-h)$ is a non-increasing set function.}
\end{mydef}

{Therefore, a function $h:2^\mathcal{V}\mapsto \mathbb{R}$ is submodular if, for any $a\in \mathcal{V}$, the function $h_a:2^{\mathcal{V}\setminus\{a\}}\mapsto \mathbb{R}$, defined as $h_a(\Delta)\equiv h(\Delta\cup \{a\})-h(\Delta)$, is a non-increasing set function.  This property is also called the \emph{diminishing returns property}.}

{Next, we present a fact from the supermodular functions minimization literature, that we use in Section \ref{sec:Min_N} so as to construct an approximation algorithm for Problem \eqref{pr:min_set}. In particular, consider the following optimization program, which is of similar structure to \eqref{pr:min_set}, where $h:2^\mathcal{V}\mapsto \mathbb{R}$ is a non-decreasing, supermodular set function:}

\begin{equation}\tag{$\mathcal{O}$} \label{pr:X}
\begin{aligned}
 \underset{\Delta \subseteq \mathcal{V}}{\text{minimize}} & \; \;  \; 
 |\Delta|\\
\text{subject to} \\
&h(\Delta) \leq E.
\end{aligned}
\end{equation}

{The following greedy algorithm has been proposed for its approximate solution, for which, the subsequent fact is true.}

\begin{algorithm}
\caption{Approximation Algorithm for the Problem~\eqref{pr:X}.}
\begin{algorithmic}
\REQUIRE $h$, $E$.
\ENSURE Approximate solution to Problem \eqref{pr:X}.
\STATE $\Delta\leftarrow\emptyset$
\WHILE {$h(\Delta) > E $} \STATE{ 	
    $a_i \leftarrow a'\in \arg\max_{a \in \mathcal{V}\setminus \Delta}\{
	h(\Delta)-h(\Delta\cup \{a\})
	\}$\\
	\quad \mbox{} $\Delta \leftarrow \Delta \cup \{a_i\}$	
	}
\ENDWHILE
\end{algorithmic} \label{alg:X}
\end{algorithm}

\begin{myfact}\label{fact}
{Denote as $\Delta^\star$ a solution to Problem~\eqref{pr:X} and as $\Delta_0, \Delta_1, \ldots$ the sequence of sets picked by Algorithm \ref{alg:X}.  Moreover, let $l$ be the smallest index such that $h(\Delta_l) \leq E$.  Then,
\[
\frac{l}{|\Delta^\star|} \leq 1+\log \frac{h(\mathcal{V})-h(\emptyset)}{h(\mathcal{V})-h(\Delta_{l-1})}.
\]}
\end{myfact}

{In Section \ref{sec:Min_N}, we provide an efficient approximation algorithm for \eqref{pr:min_set}, by applying Fact \ref{fact} to an appropriately perturbed version of this problem, so that it involves a non-decreasing supermodular function, as in \eqref{pr:X}.}
{This also leads to our second main contribution, presented in Section \ref{sec:Min_E}:  An efficient approximation algorithm for Problem \eqref{pr:original}, which selects all the actuators to maximize the volume of the set of states that
can be reached with one unit or less of input energy, without assuming controllability beforehand.  This is in contrast to the related works \cite{summers2016submodularity} and \cite{7039832}: there, the authors consider a similar problem for choosing a few actuators to optimize
$\log\det({W}_{(\cdot)}^{-1})$; however, their results rely on the assumption of a pre-existing actuator set that renders~\eqref{eq:dynamics} controllable, although this set is not selected towards the minimization of $\log\det({W}_{(\cdot)}^{-1})$.  Nevertheless, this assumption is necessary, since they then prove that the $\log\det({W}_{(\cdot)}^{-1})$ is a supermodular function in the choice of the extra actuators.}
{On the other hand, our algorithms select all the actuators towards the involved energy objective, since they rely on a $\epsilon$-perturbed version of $\log\det({W}_{(\cdot)}^{-1})$, that we prove to be supermodular without assuming controllability beforehand.}


{Overall, our results supplement the existing literature by considering Problems \eqref{pr:min_set} and \eqref{pr:original} when the system is not initially controllable and by providing efficient approximation algorithms for their solution, along with worst-case performance guarantees.}

\section{Minimal Actuator Sets with Constrained Control Effort}\label{sec:Min_N}

{We present} an efficient approximation algorithm for Problem~\eqref{pr:min_set}.  To this end, we first generalize the involved energy metric to an $\epsilon$-close one that remains well-defined even when the controllability matrix is not invertible.  Next, we relax~\eqref{pr:min_set} {by introducing a new program that makes use of this metric and circumvents the restrictive controllability constraint of \eqref{pr:min_set}. Moreover, we prove that for certain values of $\epsilon$ all solutions of this auxiliary problem render the system controllable.  This fact, along with the supermodularity property of the generalized metric that we establish, leads to our proposed approximation algorithm.  The discussion of its efficiency ends the analysis of~\eqref{pr:min_set}.

\subsection{An $\epsilon$-close Auxiliary Problem}\label{subsec:Min_N_aux}

Consider the following approximation to~\eqref{pr:min_set}
\begin{equation}\tag{I$'$}\label{pr:min_set_approx}
\begin{aligned}
 \underset{\Delta \subseteq \mathcal{V}}{\text{minimize}} & \; \;  \; 
 |\Delta|\\
\text{subject to} \\
&  \log\det(\tilde{W}_\Delta+\epsilon{I})^{-1} \leq \tilde{E},
\end{aligned}
\end{equation} 
where $\tilde{W}_\Delta$ is equivalent to ${W}_\Delta/(2\lambda_{\max} (W_\mathcal{V}))$, $\lambda_{\max} (W_\mathcal{V})$ is the maximum eigenvalue of $W_\mathcal{V}$, $\tilde{E}$ is equal to $E+n\log(2\lambda_{\max}(W_\mathcal{V}))$, and $\epsilon$ is positive.  

In contrast to~\eqref{pr:min_set}, the domain of this problem consists of all subsets of $\mathcal{V}$ since $\tilde{W}_{(\cdot)}+\epsilon{I}$ is always invertible.
The $\epsilon$-closeness is evident since for any
$\Delta$ such that $(A,B(\Delta))$ is controllable $\log\det(\tilde{W}_\Delta+\epsilon{I})^{-1} \leq \tilde{E}$ becomes $\log\det({W}_\Delta^{-1}) \leq {E}$ as $\epsilon\rightarrow 0$.  Due to the definition of $\tilde{W}_\Delta$, for all $\Delta \subseteq \mathcal{V}$, all eigenvalues of $\tilde{W}_\Delta$ are at most $1/2$ \cite[Theorem 8.4.9]{bernstein2009matrix}; this property will be useful in the proof of one of our main results, in particular, Proposition \ref{prop:suf_contr}.

{In the following paragraphs, we identify an approximation algorithm for solving Problem~\eqref{pr:min_set_approx}, and correspondingly, the $\epsilon$-close, NP-hard Problem~\eqref{pr:min_set}.}

\subsection{Approximation Algorithm for Problem~\eqref{pr:min_set_approx}}\label{subsec:Min_N_alg}

We first prove that all solutions of~\eqref{pr:min_set_approx} for $0<\epsilon \leq \min\{1/2, e^{-\tilde{E}}\}$ render the system controllable, notwithstanding that no controllability constraint is imposed by this program on the choice of the actuator sets.  Moreover, we show that the involved $\epsilon$-close energy metric is supermodular with respect to the choice of actuator sets and then we present our approximation algorithm, followed by a discussion of its efficiency which ends this subsection.

\begin{myproposition}\label{prop:suf_contr}
Consider a constant $\omega>0$, $\epsilon$ such that $0< \epsilon< \min\{1/2, e^{-\omega}\}$, and any $\Delta \subseteq \mathcal{V}$: If $\log\det(\tilde{W}_\Delta+\epsilon{I})^{-1} \leq \omega$, then $(A,B(\Delta))$ is controllable.
\end{myproposition}
\begin{proof}
Assume that $(A,B(\Delta))$ is not controllable and let $k$ be the corresponding number of non-zero eigenvalues of ${W}_\Delta$ which we denote as $\lambda_1, \lambda_2, \ldots, \lambda_k$; therefore, $k \leq n-1$.  Then,
\begin{align}
\begin{split}
&\log\det(\tilde{W}_\Delta+\epsilon{I})^{-1}=\sum_{i=1}^{k}\log\frac{1}{\frac{\lambda_i}{2\lambda_{\max}(W_\mathcal{V})}+\epsilon}\\
&\hspace{28mm}\qquad\qquad+(n-k)\log\frac{1}{\epsilon}>\log \frac{1}{\epsilon}> \omega, \nonumber
\end{split}
\end{align} since $\frac{\lambda_i}{2\lambda_{\max}(W_\mathcal{V})}+\epsilon < 1$ (because $\frac{\lambda_i}{2\lambda_{\max}(W_\mathcal{V})}\leq 1/2$ and $\epsilon<1/2$), and $\epsilon < e^{-\omega}$.  Therefore, we have a contradiction.
\end{proof}

Note that $\omega$ is chosen independently of the parameters of system~\eqref{eq:dynamics}.  Therefore,  the absence of the controllability constraint in Problem~\eqref{pr:min_set_approx} for $0<\epsilon \leq \min\{1/2, e^{-\tilde{E}}\}$ is fictitious; nonetheless, it obviates the necessity of considering only actuator sets that render the system controllable.

The next proposition is also essential and suggests an efficient approximation algorithm for solving~\eqref{pr:min_set_approx}.

\begin{myproposition}[Supermodularity]\label{prop:subm}
The function $\log\det(\tilde{W}_\Delta+\epsilon{I})^{-1}: \Delta \subseteq \mathcal{V} \mapsto \mathbb{R}$ is
supermodular and non-increasing set with respect to the choice of $\Delta$.
\end{myproposition}
\begin{proof}
To prove that the $\log\det(\tilde{W}_\Delta+\epsilon{I})^{-1}$ is non-increasing, recall from \eqref{eq:gramianTOdelta}
that for any $\Delta_1\subseteq\Delta_2\subseteq [n]$, $\tilde{W}_{\Delta_1} \preceq \tilde{W}_{\Delta_2}$.  Therefore, from \cite[Theorem 8.4.9]{bernstein2009matrix},
$
\log\det(\tilde{W}_{\Delta_2}+\epsilon{I})^{-1} \preceq
\log\det(\tilde{W}_{\Delta_1}+\epsilon{I})^{-1},
$
and as a result, $\log\det(\tilde{W}_\Delta+\epsilon{I})^{-1}$ is non-increasing. 

Next, to prove that $\log\det(\tilde{W}_\Delta+\epsilon{I})^{-1}$ is a supermodular set function, recall from Section \ref{subsec:supermod} that it suffices to prove that $\log\det(\tilde{W}_\Delta+\epsilon{I})$ is a submodular one.  
In particular, recall that a function $h: 2^{[n]}\mapsto \mathbb{R}$ is submodular if and only if, for any $ a \in [n]$, the function $h_a: 2^{[n]\setminus\{a\}} \mapsto \mathbb{R}$, where $h_a(\Delta)\equiv h(\Delta\cup \{a\})-h(\Delta)$, is a non-increasing set function.  Therefore, to prove that  $h(\Delta)=\log\det(\tilde{W}_\Delta+\epsilon{I})$ is submodular, we may prove that the $h_a(\Delta)$ is a non-increasing set function.  To this end, we follow the proof of Theorem 6 in \cite{summers2016submodularity}: first, observe that
\begin{align*}
h_a(\Delta)&= \log\det(\tilde{W}_{\Delta\cup \{a\}}+\epsilon I)-\log\det(\tilde{W}_\Delta+\epsilon I)\\
&=\log\det(\tilde{W}_\Delta+\tilde{W}_{a}+\epsilon I)-\log\det(\tilde{W}_\Delta+\epsilon I).
\end{align*}
Now, for any $\Delta_1\subseteq\Delta_2\subseteq [n]$ and $z \in [0,1]$, define $\Omega(z)\equiv \epsilon I+\tilde{W}_{\Delta_1}+z(\tilde{W}_{\Delta_2}-\tilde{W}_{\Delta_1})$ and 
$
\bar{h}(z)\equiv \log\det(\Omega(z)+\tilde{W}_{a})-\log\det\left(\Omega(z)\right);
$
it is $\bar{h}(0)=h_a(\Delta_1)$ and $\bar{h}(1)=h_a(\Delta_2)$.  Moreover, since $d\log\det(\Omega(z)))/dz=\text{tr}\left(\Omega(z)^{-1} d\Omega(z)/dz\right)$ (cf. equation (43) in \cite{petersen2008matrix}),
\begin{align*}
\frac{d\bar{h}(z)}{dz}= \text{tr}[((\Omega(z)+\tilde{W}_{a})^{-1}-\Omega(z)^{-1})O_{21}],
\end{align*}
where $O_{21}\equiv \tilde{W}_{\Delta_2}-\tilde{W}_{\Delta_1}$.  From \cite[Proposition 8.5.5]{bernstein2009matrix}, 
$
(\Omega(z)+\tilde{W}_{a})^{-1} \preceq \Omega(z)^{-1},
$
because $\Omega(z)\succ 0$ for all $z \in [0,1]$, since $\epsilon I \succ 0$, $\tilde{W}_{\Delta_1} \succeq 0$, and $\tilde{W}_{\Delta_2}\succeq \tilde{W}_{\Delta_1}$.  Thereby, from \cite[ Corollary 8.3.6]{bernstein2009matrix}, all eigenvalues of
$
((\Omega(z)+\tilde{W}_{a})^{-1}-\Omega(z)^{-1})O_{21}$ are non-positive.
As a result, ${d\bar{h}(z)}/{dz}\leq 0$, and
\begin{align*}
h_a(\Delta_2)=\bar{h}(1)=\bar{h}(0)+\int_{0}^{1}\frac{d\bar{h}(z)}{dz}dz\leq \bar{h}(0)=h_a(\Delta_1).
\end{align*}
Therefore, $h_a(\Delta)$ is a non-increasing set function, and the proof is complete.
\end{proof}

{Therefore, the hardness of the $\epsilon$-close Problem \eqref{pr:min_set} is in agreement with that of the class of minimum set-covering problems subject to submodular constraints.}  Inspired by this literature~\cite{Nemhauser:1988:ICO:42805, citeulike:416650,krause2012submodular}, we have the following efficient approximation algorithm for Problem~\eqref{pr:min_set_approx}, and as we show by the end of this section, for Problem~\eqref{pr:min_set} as well.

\begin{algorithm}
\caption{Approximation Algorithm for the Problem~\eqref{pr:min_set_approx}.}\label{alg:minimal-leaders}
\begin{algorithmic}
\REQUIRE Bound $\tilde{E}$, parameter $\epsilon \leq \min\{1/2,e^{-\tilde{E}}\}$, matrices ${W}_1, {W}_2, \ldots,$ ${W}_n$.
\ENSURE Actuator set $\Delta$.
\STATE $\Delta\leftarrow\emptyset$
\WHILE {$\log\det(\tilde{W}_\Delta+\epsilon{I})^{-1} > \tilde{E} $} \STATE{ 	
    $a_i {\leftarrow a'\in} \arg\max_{a \in \mathcal{V}\setminus \Delta}\{
	\log\det(\tilde{W}_\Delta+\epsilon{I})^{-1}-\log\det(\tilde{W}_{\Delta\cup \{a\}}+\epsilon{I})^{-1}
	\}$\\
	\qquad\hspace{-1mm} $\Delta \leftarrow \Delta \cup \{a_i\}$	
	}
\ENDWHILE
\end{algorithmic} 
\end{algorithm}

Regarding the quality of Algorithm~\ref{alg:minimal-leaders} the following is true.

\begin{mytheorem}[A Submodular Set Coverage Optimization]\label{th:minimal}
Denote as $\Delta^\star$ a solution to Problem~\eqref{pr:min_set_approx} and as $\Delta$ the selected set by Algorithm~\ref{alg:minimal-leaders}.  Then,
\begin{align}
&(A,B(\Delta)) \text{ is controllable},\label{explain:th:minimal2}\\
&\log\det(\tilde{W}_\Delta+\epsilon{I})^{-1} \leq \tilde{E},\label{explain:th:minima3}\\
&\frac{|\Delta|}{|\Delta^\star|}\leq 1+\log \frac{n\log(\epsilon^{-1})-\log\det(\tilde{W}_\mathcal{V}+\epsilon{I})^{-1}}{\tilde{E}-\log\det(\tilde{W}_\mathcal{V}+\epsilon{I})^{-1}}\equiv F, \label{explain:th:minimal1}\\
&F = O(\log n+\log\log(\epsilon^{-1})+\log \frac{1}{\tilde{E}-\log\det(\tilde{W}_\mathcal{V}^{-1})}).
\label{explain:approx_error0}
\end{align}
{Finally, the computational complexity of Algorithm \ref{alg:minimal-leaders} is $O(n^5)$.}
\end{mytheorem}
\begin{proof}
We first prove~\eqref{explain:th:minima3},~\eqref{explain:th:minimal1} and~\eqref{explain:approx_error0}, and then, \eqref{explain:th:minimal2}.  {We end the proof by clarifying the computational complexity of Algorithm \ref{alg:minimal-leaders}.} 

First, let $\Delta_0, \Delta_1, \ldots$ be the sequence of sets selected by Algorithm~\ref{alg:minimal-leaders} and $l$ the smallest index such that $\log\det(\tilde{W}_{\Delta_l}+\epsilon{I})^{-1}  \leq E$.  Therefore, $\Delta_l$ is the set that Algorithm~\ref{alg:minimal-leaders} returns, and this proves~\eqref{explain:th:minima3}.

Moreover, from~\cite{citeulike:416650}, since for any $\Delta \subseteq \mathcal{V}$, $h(\Delta)\equiv-\log\det(\tilde{W}_\Delta+\epsilon{I})^{-1} +n\log(\epsilon^{-1})$ is a non-negative, non-decreasing, and submodular function (cf.~Proposition~\ref{prop:subm}), it is guaranteed for Algorithm~\ref{alg:minimal-leaders} that (cf.~Fact \ref{fact})
\begin{align*}
\frac{l}{|\Delta^\star|}&\leq 1+\log \frac{h(\mathcal{V})-h(\emptyset)}{h(\mathcal{V})-h(\Delta_{l-1})}\\
&=1+\\
&\log \frac{n\log(\epsilon^{-1})-\log\det(\tilde{W}_\mathcal{V}+\epsilon{I})^{-1}}{\log\det(\tilde{W}_{\Delta_{l-1}}+\epsilon{I})^{-1}-\log\det(\tilde{W}_\mathcal{V}+\epsilon{I})^{-1}}.
\end{align*}
Now, $l$ is the first time that $\log\det(\tilde{W}_{\Delta_l}+\epsilon{I})^{-1} \leq \tilde{E}$, and a result $\log\det(\tilde{W}_{\Delta_{l-1}}+\epsilon{I})^{-1} > \tilde{E}$.  This implies~\eqref{explain:th:minimal1}.  

Moreover, observe that $0<\log\det(\tilde{W}_\mathcal{V}+\epsilon{I})^{-1}<\log\det(\tilde{W}_\mathcal{V}^{-1})$ so that from \eqref{explain:th:minimal1} we get $F \leq 1+\log [n\log(\epsilon^{-1})/(\tilde{E}-\log\det(\tilde{W}_\mathcal{V}^{-1}))]$,
which in turn implies~\eqref{explain:approx_error0}.

On the other hand, since $0<\epsilon \leq \min\{1/2,e^{-\tilde{E}}\}$ and $\log\det(\tilde{W}_{\Delta_l}+\epsilon{I})^{-1} \leq \tilde{E}$, Proposition~\ref{prop:suf_contr} is in effect, i.e.,~\eqref{explain:th:minimal2} holds true.

{Finally, with respect to the computational complexity of Algorithm \ref{alg:minimal-leaders}, note that the \texttt{while} loop is repeated for at most $n$ times.  Moreover, the complexity to compute the determinant an $n \times n$ matrix, using Gauss-Jordan elimination decomposition, is $O(n^3)$. Additionally, at most $n$ matrices must be inverted so that the ``$\arg\max_{a \in \mathcal{V}\setminus \Delta}\{
	\log\det(\tilde{W}_\Delta+\epsilon{I})^{-1}-\log\det(\tilde{W}_{\Delta\cup \{a\}}+\epsilon{I})^{-1}
	\}$'' can be computed.  Furthermore, $O(n)$ time is required to find a maximum element between $n$ available.  Therefore, the computational complexity of Algorithm \ref{alg:minimal-leaders} is $O(n^5)$.}
\end{proof}

Therefore, Algorithm~\ref{alg:minimal-leaders} returns a set of actuators that meets the corresponding control energy bound of Problem~\eqref{pr:min_set_approx} while it renders system~\eqref{eq:dynamics} controllable.  Moreover, the cardinality of this set is up to a multiplicative factor of $F$ from the minimum cardinality actuator sets that meet the same control energy bound. 

The dependence of $F$ on $n,\epsilon$ and $E$ was expected from a design perspective: Increasing the network size $n$ or improving the accuracy by decreasing $\epsilon$, as well as demanding a better energy guarantee by decreasing $E$ should all push the cardinality of the selected actuator set upwards.  Also, $\log\log(\epsilon^{-1})$ is the design cost for circumventing the difficult to satisfy controllability constraint of~\eqref{pr:min_set}~\cite{2013arXiv1304.3071O}, i.e., {for assuming no pre-existing actuators that renders~\eqref{eq:dynamics} controllable and choosing all the actuators towards the satisfaction of an energy performance criterion}.

{From a computational perspective, the computation of the determinant is the only intensive procedure of Algorithm \ref{alg:minimal-leaders}, requiring $O(n^3)$ time, if we use the Gauss-Jordan elimination decomposition.  On the other hand, to apply this algorithm on large-scale systems, we can speed up this procedure using the Coppersmith-Winograd algorithm \cite{coppersmith1987matrix}, which requires $O(n^{2.376})$ time. Alternatively, we can use numerical methods, which efficiently compute an approximate the determinant of a matrix even if its size is of several thousands \cite{Reusken:2001:ADL:587707.587812}.  Moreover, we can speed up Algorithm \ref{alg:minimal-leaders} using a method proposed in \cite{minoux1978accelerated}, which avoids the computation of $\log\det(\tilde{W}_\Delta+\epsilon{I})^{-1}-\log\det(\tilde{W}_{\Delta\cup \{a\}}+\epsilon{I})^{-1}$ for unnecessary choices of $a$, towards the computation of the $\arg\max_{a \in \mathcal{V}\setminus \Delta}\{
	\log\det(\tilde{W}_\Delta+\epsilon{I})^{-1}-\log\det(\tilde{W}_{\Delta\cup \{a\}}+\epsilon{I})^{-1}
	\}$, by taking advantage of the supermodularity of $\log\det(\tilde{W}_{(\cdot)}+\epsilon{I})^{-1}$.}

{Finally, for large values of $n$, the computation of ${W}_1, {W}_2, \ldots,$ ${W}_n$ is demanding as well.  On the other hand, in the case of stable systems, as many physical, e.g., biological, networks are, the corresponding controllability Gramians can be used instead, which for a stable system can be calculated from the Lyapunov equations ${A}{G_i} + {G_i}{A}^{T} = -I^{(i)}$, for $i=1,2,\ldots,n$, respectively, and are given in closed-form by
\begin{align}
{G_i} = \int_{t_0}^{\infty} \mathrm{e}^{{A}(t-t_0)} I^{(i)} \mathrm{e}^{{A}^{T}(t-t_0)}\,\mathrm{d}{t}. \label{eq:lyap}
\end{align}
Using these Gramians for the evaluation of $W$ in \eqref{eq:min_energy} corresponds to the minimum state transfer energy with no time constraints.  The advantage of this approach is that \eqref{eq:lyap} can be solved efficiently using numerical methods, even when the system's size $n$ has a value of several thousands \cite{benner2008numerical}.}

In Section~\ref{subsec:ApproximationAlgorithm} we finalize our treatment of Problem~\eqref{pr:min_set} by employing Algorithm~\ref{alg:minimal-leaders} to approximate its solutions.

\subsection{Approximation Algorithm for Problem~\eqref{pr:min_set}}\label{subsec:ApproximationAlgorithm}

We present an efficient approximation algorithm for Problem~\eqref{pr:min_set} that is based on Algorithm~\ref{alg:minimal-leaders}.  Let $\Delta$ be the actuator set returned by Algorithm~\ref{alg:minimal-leaders}, so that $(A,B(\Delta))$ is controllable and $\log\det(\tilde{W}_{\Delta}+\epsilon{I})^{-1}\leq \tilde{E}$.   For any $c>0$, there exists sufficiently small $\epsilon(c)$ such that: 
\begin{align}
\log\det(\tilde{W}_{\Delta}+\epsilon(c) I)^{-1}&\geq \log\det(\tilde{W}_{\Delta}^{-1})-c\tilde{E}.\label{eq:aux_3}
\end{align}
 Moreover, $\log\det(\tilde{W}_{\Delta}+\epsilon(c) I)^{-1}\leq \tilde{E}$, and therefore we get from~\eqref{eq:aux_3} that $\log\det(\tilde{W}_{\Delta}^{-1}) \leq (1+c)\tilde{E}$, or
\begin{align}
\log\det({W}_{\Delta}^{-1}) \leq E+c\tilde{E}. \label{eq:approx_error}
\end{align} Hence, we refer to $c$ as \textit{approximation error}.  

On the other hand, $\epsilon(c)$ is not known a priori.   Hence, we need to search for a sufficiently small $\epsilon$ so that~\eqref{eq:approx_error} holds true.  One way to achieve this since $\epsilon$ is lower and upper bounded by $0$ and $\min\{1/2,e^{-\tilde{E}}\}$, respectively, is to perform a search using bisection.  We implement this procedure in Algorithm~\ref{alg:minimal-leaders_final}, where we denote as  $[\text{Algorithm}~\ref{alg:minimal-leaders}](\tilde{E},\epsilon)$ the set that Algorithm~\ref{alg:minimal-leaders} returns for given $\tilde{E}$ and $\epsilon$.

\begin{algorithm}[h]
\caption{Approximation Algorithm for the Problem~\eqref{pr:min_set}.}\label{alg:minimal-leaders_final}
\begin{algorithmic}
\REQUIRE Bound ${E}$, approximation error $c$, bisection's initial accuracy level $a_0$, matrices ${W}_1, {W}_2, \ldots, {W}_n$.
\ENSURE Actuator set $\Delta$.
\STATE $a \leftarrow a_0$, $\text{flag}\leftarrow 0$, $l\leftarrow 0$, $u\leftarrow \min\{1/2,e^{-\tilde{E}}\}$, $\epsilon\leftarrow(l+u)/2$
\WHILE {$\text{flag} \neq 1$}
	\WHILE {$u-l>a$}\\	
	\quad $\Delta \leftarrow [\text{Algorithm}~\ref{alg:minimal-leaders}](\tilde{E},\epsilon)$
		\IF {$\log\det(\tilde{W}_{\Delta}^{-1})- \log\det(\tilde{W}_{\Delta}+\epsilon{I})^{-1}> c\tilde{E}$} \STATE{$u\leftarrow \epsilon$} \ELSE \STATE{$l\leftarrow \epsilon$} 
		\ENDIF\\
	\quad $\epsilon\leftarrow (l+u)/2$   
	\ENDWHILE
\IF {$\log\det(\tilde{W}_{\Delta}^{-1})- \log\det(\tilde{W}_{\Delta}+\epsilon{I})^{-1}> c\tilde{E}$} \STATE{$u\leftarrow \epsilon$}, $\epsilon\leftarrow (l+u)/2$
\ENDIF\\
$\Delta \leftarrow [\text{Algorithm}~\ref{alg:minimal-leaders}](\tilde{E},\epsilon)$\label{exit_step}
\IF {$\log\det(\tilde{W}_{\Delta}^{-1})- \log\det(\tilde{W}_{\Delta}+\epsilon{I})^{-1} \leq c\tilde{E}$} \STATE{$\text{flag} \leftarrow 1$} \ELSE \STATE {$a\leftarrow a/2$}
\ENDIF
\ENDWHILE
\end{algorithmic} 
\end{algorithm}

In the worst case, when we first enter the inner \texttt{while} loop, the \texttt{if} condition is not satisfied, and as a result $\epsilon$ is set to a lower value.  This process continues until the \texttt{if} condition is satisfied for the first time{, given that $a_0$ is sufficiently small for the specified $c$,} from which point and on this \texttt{while} loop converges up to the accuracy level $a_0$ to the largest value $\bar{\epsilon}$ of $\epsilon$ such that $\log\det(\tilde{W}_{\Delta}^{-1})- \log\det(\tilde{W}_{\Delta}+\epsilon{I})^{-1} \leq c\tilde{E}$; specifically, $|\epsilon-\bar{\epsilon}| \leq a_0/2$, due to the mechanics of the bisection method.  {On the other hand, if $a_0$ is not sufficiently small, the value of $a$ decreases within the last \texttt{if} statement of the algorithm, the variable $\text{flag}$ remains zero and the outer loop is executed again, until the convergence within the inner \texttt{while} is feasible. Then,} the \texttt{if} statement that follows the inner \texttt{while} loop ensures that $\epsilon$ is set below $\bar{\epsilon}$, so that $\log\det(\tilde{W}_{\Delta}^{-1})- \log\det(\tilde{W}_{\Delta}+\epsilon{I})^{-1} \leq c\tilde{E}$. Finally, the last \texttt{if} statement sets the $\text{flag}$ to $1$ and the algorithm terminates.  The efficiency of this algorithm for Problem~\eqref{pr:min_set} is summarized below.

\begin{mytheorem}[Approximation Efficiency and Computational Complexity of Algorithm~\ref{alg:minimal-leaders_final} for Problem~\eqref{pr:min_set}]\label{th:minimal_set_main}
Denote as $\Delta^\star$ a solution to Problem~\eqref{pr:min_set} and as $\Delta$ the selected set by Algorithm~\ref{alg:minimal-leaders_final}.  Then,
\begin{align}
&(A,B(\Delta)) \text{ is controllable},\nonumber\\
&\log\det({W}_{\Delta}^{-1}) \leq E+c\tilde{E}, \label{eq:cor_1_1}\\ 
&\frac{|\Delta|}{|\Delta^\star|}\leq  F,\label{eq:cor_1_2}\\
&F = O(\log n+ \max\{\log\log(n/(c\tilde{E})),\log\tilde{E}\}+
\nonumber\\
&\hspace{40mm}\log \frac{1}{\tilde{E}-\log\det(\tilde{W}_\mathcal{V}^{-1})})\label{eq:cor_1_3}.
\end{align}
{Finally, let $a$ be the bisection's accuracy level that Algorithm \ref{alg:minimal-leaders_final} terminates with.  Then, if $a=a_0$, the computational complexity of Algorithm \ref{alg:minimal-leaders_final} is  $O(n^5\log_2(1/a_0)$, else it is $O(n^5\log_2(1/a)\log_2(a_0/a))$.}
\end{mytheorem}
\begin{proof} We only prove statements \eqref{eq:cor_1_1}, \eqref{eq:cor_1_2} and \eqref{eq:cor_1_3}, while the first follows from Theorem~\ref{th:minimal}. {We end the proof by clarifying the computational complexity of Algorithm \ref{alg:minimal-leaders_final}.} 

First, when Algorithm~\ref{alg:minimal-leaders_final} exits the \texttt{while} loop, and after the following \texttt{if} statement, $\log\det(\tilde{W}_{\Delta}^{-1})- \log\det(\tilde{W}_{\Delta}+\epsilon{I})^{-1} \leq c\tilde{E}$, and since $\log\det(\tilde{W}_{\Delta}+\epsilon{I})^{-1}\leq \tilde{E}$, this implies~\eqref{eq:cor_1_1}.

To show~\eqref{eq:cor_1_2}, consider any solution $\Delta^\star$ to Problem~\eqref{pr:min_set} and any solution $\Delta^\bullet$ to Problem~\eqref{pr:min_set_approx}. Then, $|\Delta^\star| \geq |\Delta^\bullet|$; to see this, note that for any $\Delta^\star$, $\log\det(\tilde{W}_{\Delta^\star}+\epsilon{I})^{-1} <\log\det(\tilde{W}_{\Delta^\star}^{-1})\leq \tilde{E}$ since $\epsilon>0$, i.e., $\Delta^\star$ is a candidate solution to Problem~\eqref{pr:min_set_approx} because it satisfies all of its constraints.  Therefore, $|\Delta^\star| \geq |\Delta^\bullet|$, and as a result
$|\Delta|/|\Delta^\star|\leq |\Delta|/|\Delta^\bullet|\leq F$ per~\eqref{explain:th:minimal1}.

Next, note that~\eqref{eq:cor_1_1} holds true when, e.g., $\epsilon$ is equal to $c\tilde{E}/(2n)$.  Therefore, since also $\epsilon \leq e^{-\tilde{E}}$, $\log\log\epsilon^{-1}=O(\max\{\log\log(n/(c\tilde{E})),\log\tilde{E}\})$ and this proves~\eqref{eq:cor_1_3}.

{Finally, with respect to the computational complexity of Algorithm \ref{alg:minimal-leaders_final}, note that the inner \texttt{while} loop is repeated for at most $\log_2(1/(2a))$ times (since $\epsilon \leq 1/2$), in the worst case.  Moreover, the time complexity of the procedures within this loop is of order $O(n^5)$, due to Algorithm \ref{alg:minimal-leaders}.  Finally, if $a=a_0$, the outer \texttt{while} loop runs for one time, and otherwise, for $\log_2(a_0/a)$ times.  Therefore, the computational complexity of Algorithm \ref{alg:minimal-leaders_final} is $O(n^5\log_2(1/a_0))$, or $O(n^5\log_2(1/a)\log_2(a_0/a))$, respectively.}
\end{proof}

{From a computational perspective, we can speed up Algorithm \ref{alg:minimal-leaders_final} using the methods we discussed in the end of Section \ref{subsec:Min_N_alg}.  Moreover, for a wide class of systems, e.g., when $a=O(n^{n^{c_1}})$, where $c_1$ is a positive constant, independent of $n$, this algorithm runs in polynomial time, due to the logarithmic dependence on $a$.}

{From an approximation efficiency perspective we have that $F=O(\log(n))$, whenever $E=O(n^{c_1})$, $\lambda_{\max}(W_\mathcal{V})=O(n^{n^{c_2}})$ and $1/(\tilde{E}-\log\det(\tilde{W}_\mathcal{V}^{-1}))=O(n^{c_3})$, where $c_1$, $c_2$ and $c_3$ are positive constants and independent of $n$. In other words, the cardinality of the actuator set that Algorithm~\ref{alg:minimal-leaders_final} returns is up to a multiplicative factor of $O(\log n)$ from the minimum cardinality actuator sets that meet the same energy bound. Indeed, this is the best achievable bound in polynomial time for the set covering problem in the worst case~\cite{Feige:1998:TLN:285055.285059}, while~\eqref{pr:min_set} is a generalization of it~\cite{2013arXiv1304.3071O}.  
Thus, Algorithm~\ref{alg:minimal-leaders_final} is a best-approximation of \eqref{pr:min_set} for this class of systems.}



\section{Minimum Energy Control by a Cardinality-Constrained Actuator Set}\label{sec:Min_E}

We present an approximation algorithm for Problem~\eqref{pr:original} following a parallel line of thought as in Section~\ref{sec:Min_N}: First, we circumvent the  restrictive controllability constraint of~\eqref{pr:original} using the $\epsilon$-close generalized energy metric defined in Section~\ref{sec:Min_N}.  Then, we propose an efficient approximation algorithm for its solution that makes use of Algorithm~\ref{alg:minimal-leaders_final}; this algorithm returns an actuator set that always renders~\eqref{eq:dynamics} controllable while it guarantees a value for~\eqref{pr:original} that is provably close to its optimal one.  We end the analysis of~\eqref{pr:original} by explicating further the efficiency of this procedure.

\subsection{An $\epsilon$-close Auxiliary Problem}\label{subsec:Min_E_aux}

For $\epsilon>0$ consider the following approximation to~\eqref{pr:original}
\begin{equation}\tag{II$'$}\label{pr:original_approx}
\begin{aligned}
 \underset{\Delta \subseteq \mathcal{V}}{\text{minimize}} & \; \;  \; \log\det(\tilde{W}_\Delta+\epsilon{I})^{-1} \\
\text{subject to} \\
&  |\Delta| \leq r.
\end{aligned}
\end{equation} 
In contrast to~\eqref{pr:original}, the domain of this problem consists of all subsets of $\mathcal{V}$ since $\tilde{W}_{(\cdot)}+\epsilon{I}$ is always invertible.  Moreover, its objective is $\epsilon$-close to that of Problem~\eqref{pr:original}.  

{In the following paragraphs, we identify an efficient approximation algorithm for solving Problem~\eqref{pr:original_approx}, and correspondingly, the $\epsilon$-close, NP-hard Problem~\eqref{pr:original}.} We note that the hardness of the latter is in accordance with that of the general class of supermodular function minimization problems, as per Proposition~\ref{prop:subm} the objective $\log\det(\tilde{W}_\Delta+\epsilon{I})^{-1}$ is supermodular.  The approximation algorithms used in that literature however~\cite{Nemhauser:1988:ICO:42805, citeulike:416650,krause2012submodular}, fail to provide an efficient solution algorithm for~\eqref{pr:original_approx} --- for completeness, we discuss this direction in the Appendix~\ref{appen:badalg}.  In the next subsection we propose an efficient approximation algorithm for \eqref{pr:original} that makes use of Algorithm~\ref{alg:minimal-leaders_final}.

\subsection{Approximation Algorithm for Problem~\eqref{pr:original}}\label{subsec:Min_E_alg}

We provide an efficient approximation algorithm for Problem~\eqref{pr:original} that is based on Algorithm~\ref{alg:minimal-leaders_final}.  In particular, since~\eqref{pr:original} finds an actuator set that minimizes $\log\det({W}_{(\cdot)}^{-1})$, and any solution to~\eqref{pr:min_set} satisfies $\log\det({W}_{(\cdot)}^{-1})\leq E$, one may repeatedly execute Algorithm~\ref{alg:minimal-leaders_final} for decreasing values of ${E}$ as long as the returned actuators are at most $r$ and ${E}$ satisfies the feasibility constraint ${E}\geq \log\det({W}_\mathcal{V}^{-1})$ (cf.~Section~\ref{subsec:problems}).  Therefore, for solving~\eqref{pr:original} we propose a bisection-type execution of Algorithm~\ref{alg:minimal-leaders_final} with respect to $E$.  

To this end, we also need an upper bound for the value of~\eqref{pr:original}:  Let $\Delta_\mathcal{C}$ be a small actuator set that renders system~\eqref{eq:dynamics} controllable; it is efficiently found using Algorithm~\ref{alg:minimal-leaders_final} for large $E$ or the procedure proposed in~\cite{2013arXiv1304.3071O}.  Then, for any $r\geq |\Delta_\mathcal{C}|$,  $\log\det(\tilde{W}_{\Delta_\mathcal{C}}^{-1})$ upper bounds the value of~\eqref{pr:original} since $\log\det(\tilde{W}_{(\cdot)}^{-1})$ is monotone.  

Thus, having a lower and upper bound for the value of~\eqref{pr:original}, we implement Algorithm~\ref{alg:min_energy_2} for approximating the solutions of~\eqref{pr:original}; we consider only the non-trivial case where $r<n$ and denote the set that Algorithm~\ref{alg:minimal-leaders_final} returns as $[\text{Algorithm}~\ref{alg:minimal-leaders_final}](\tilde{E},c,a_0)$ for given $\tilde{E}$, $c$ and $a_0$.

\begin{algorithm}
\caption{Approximation algorithm for Problem~\eqref{pr:original}.}\label{alg:min_energy_2}
\begin{algorithmic}
\REQUIRE Set $\Delta_\mathcal{C}$, maximum number of actuators $r$ such that $r\geq |\Delta_\mathcal{C}|$,  approximation error $c$ for Algorithm~\ref{alg:minimal-leaders_final}, bisection's accuracy level $a_0$ for Algorithm~\ref{alg:minimal-leaders_final}, bisection's accuracy level $a_0'$ for current algorithm, matrices ${W}_1, {W}_2, \ldots, {W}_n$.
\ENSURE Actuator set $\Delta$.
\STATE $\Delta\leftarrow\emptyset$,
$l\leftarrow \log\det(\tilde{W}_{\mathcal{V}}^{-1})$, $u\leftarrow \text{tr}({W}_{\Delta_\mathcal{C}}^{-1})$,
$\tilde{E}\leftarrow (l+u)/2$, $\epsilon \leftarrow \min\{1/2,e^{-\tilde{E}}\}$ 
\WHILE {$ u-l >a_0'$}\\
    $\Delta\leftarrow [\text{Algorithm}~\ref{alg:minimal-leaders_final}](\tilde{E},c,a_0)$
	\IF {$|\Delta| > r$} \STATE{$l\leftarrow \tilde{E}$, $\tilde{E}\leftarrow (l+u)/2$} \ELSE \STATE{$u\leftarrow \tilde{E}$, $\tilde{E}\leftarrow (l+u)/2$} 
	\ENDIF\\
	$\epsilon \leftarrow 1/\tilde{E}$    
\ENDWHILE
\IF {$|\Delta| > r$} \STATE{$l\leftarrow \tilde{E}$, $\tilde{E}\leftarrow (l+u)/2$} 
\ENDIF\\
$\Delta\leftarrow [\text{Algorithm}~\ref{alg:minimal-leaders_final}](\tilde{E},c,a_0)$
\end{algorithmic} 
\end{algorithm}

In the worst case, when we first enter the \texttt{while} loop, the \texttt{if} condition is not satisfied, and as a result $\tilde{E}$ is set to a greater value.  This process continues until the \texttt{if} condition is satisfied for the first time from which point and on the algorithm converges up to the accuracy level $a_0$ to the smallest value $\underline{\tilde{E}}$ of $\tilde{E}$ such that $|\Delta| \leq r$; specifically, $|\tilde{E}-\underline{\tilde{E}}| \leq a_0'/2$ due to the mechanics of the bisection method, where $\underline{\tilde{E}}\equiv\text{min}\{\tilde{E}:|[\text{Algorithm}~\ref{alg:minimal-leaders_final}](\tilde{E},c,a_0)|\leq r\}$. Hereby $\underline{\tilde{E}}$ is the least bound $\tilde{E}$ for which Algorithm~\ref{alg:minimal-leaders_final} returns an actuator set of cardinality at most $r$ for the specified $c$ and $a_0$ --- $\underline{\tilde{E}}$ may be larger than the value of~\eqref{pr:original} due to worst-case approximability of the involved problems (cf.~Theorem~\ref{th:minimal_set_main}).  Then, Algorithm~\ref{alg:min_energy_2} exits the \texttt{while} loop and the last \texttt{if} statement ensures that $\tilde{E}$ is set below $\underline{\tilde{E}}$ so that $|\Delta| \leq r$.  Moreover, per Theorem~\ref{th:minimal_set_main} this set renders~\eqref{eq:dynamics} controllable and guarantees that $\log\det(\tilde{W}_{\Delta}^{-1}) \leq E+c\tilde{E}$.  {Finally, with respect to the computational complexity of Algorithm \ref{alg:min_energy_2}, note that the \texttt{while} loop is repeated for at most $\log_2\left[(\log\det(\tilde{W}_{\Delta_\mathcal{C}}^{-1})-\log\det(\tilde{W}_{\mathcal{V}}^{-1}))/a_0'\right]$ times.  Moreover, the time complexity of the procedures within this loop are, in the worst case, of the same order as that of Algorithm \ref{alg:minimal-leaders_final} when it is executed for $\tilde{E}$ equal to $\underline{\tilde{E}}$. Regarding Theorem \ref{th:minimal_set_main}, denote this time complexity as $C(\underline{\tilde{E}},c, a_0)$.  Therefore, the computational complexity of Algorithm \ref{alg:minimal-leaders_final} is $O\left(C(\underline{\tilde{E}},c, a_0)\log_2\left[(\log\det(\tilde{W}_{\Delta_\mathcal{C}}^{-1})-\log\det(\tilde{W}_{\mathcal{V}}^{-1}))/a'\right]\right)$.}

We summarize the above in the next corollary, which also ends the analysis of Problem~\eqref{pr:original}.

\begin{mycorollary}[Approximation Efficiency and Computational Complexity of Algorithm~\ref{alg:min_energy_2} for Problem~\eqref{pr:original}]
Denote as $\Delta$ the selected set by Algorithm~\ref{alg:min_energy_2}. Then,
\begin{align*}
&(A,B(\Delta)) \text{ is controllable},\\
&\log\det({W}_{\Delta}^{-1}) \leq E+c\tilde{E},\\
&|\tilde{E}-\underline{\tilde{E}}|\leq a'/2,
\end{align*} where $\underline{\tilde{E}}=\text{min}\{\tilde{E}:|[\text{Algorithm}~\ref{alg:minimal-leaders_final}](\tilde{E},c,a)|\leq r\}$ is the least bound $\tilde{E}$ that Algorithm~\ref{alg:minimal-leaders_final} satisfies with an actuator set of cardinality at most $r$ for the specified $c$ and $a$. {Finally, the computational complexity of Algorithm \ref{alg:min_energy_2} is 
\[
O\left(C(\underline{\tilde{E}},c, a_0)\log_2\left(\frac{\log\det(\tilde{W}_{\Delta_\mathcal{C}}^{-1})-\log\det(\tilde{W}_{\mathcal{V}}^{-1})}{a'}\right)\right),
\]
where $C(\underline{\tilde{E}},c, a_0)$ denotes the computational complexity of Algorithm \ref{alg:minimal-leaders_final}, with respect to Theorem \ref{th:minimal_set_main}, when it is executed for $\tilde{E}$ equal to $\underline{\tilde{E}}$.}
\end{mycorollary}

{From a computational perspective, we can speed up Algorithm \ref{alg:min_energy_2} using the methods we discussed in the end of Section \ref{subsec:Min_N_alg}.  Moreover, for a wide class of systems, e.g., when $a=O(n^{n^{c_1}})$, where $c_1$ is a positive constant, independent of $n$, and similarly for $a'$ and $\log\det(\tilde{W}_{\Delta_\mathcal{C}}^{-1})$, this algorithm runs in polynomial time, due to the logarithmic dependence on $a$, $a'$ and $\log\det(\tilde{W}_{\Delta_\mathcal{C}}^{-1})$, respectively.}


\section{Concluding Remarks}\label{sec:conc}

We addressed two actuator placement problems in linear systems: First, the problem of minimal actuator placement so that the volume of the set of states reachable with one or less units of input energy is lower bounded by a desired value, and then the problem of cardinality-constrained actuator placement for minimum control effort, where the optimal actuator set is selected so that the volume of the set of states that
can be reached with one unit or less of input energy is maximized.  Both problems were shown to be NP-hard, while for the first one we provided a best approximation algorithm {{for a given range of the problem parameters}}. Next, we proposed an efficient approximation algorithm for the solution of the second problem as well.    Our future work is focused on exploring the effect that the underlying network topology of the involved system has on these actuator placement problems, as well as investigating distributed implementations of their corresponding algorithms.

\appendix

\subsection{Computational Complexity of Problems \eqref{pr:min_set} and \eqref{pr:original}}\label{appen:NP}

{We prove that Problem \ref{pr:min_set} is NP-hard, providing an instance that reduces to the NP-hard controllability problem introduced in \cite{2013arXiv1304.3071O}.   In particular, it is shown in \cite{2013arXiv1304.3071O} that deciding if \eqref{eq:dynamics} is controllable by a zero-one diagonal matrix $B$ with $r+1$ non-zero entries reduces to the $r$-hitting set problem, as we define it below, which is NP-hard \cite{arora2009computational}.}

\begin{mydef}[$r$-hitting set problem]\label{def:k_HS}
{Given a finite set $\mathcal{M}$ and a collection $\mathcal{C}$ of non-empty subsets of $\mathcal{M}$, find an $\mathcal{M'}\subseteq \mathcal{M}$ of cardinality at most $r$ that has a non-empty intersection with each set in $\mathcal{C}$.}
\end{mydef}
{Without loss of generality, we assume that every element of $\mathcal{M}$ appears in at least one set in $\mathcal{C}$ and all sets in $\mathcal{C}$ are non-empty.  Moreover in Definition \ref{def:k_HS}, we let $|\mathcal{C}|=p$ and $\mathcal{M}=\{1, 2,$ $\ldots, m\}$, and define $C\in \mathbb{R}^{p \times m}$ such that $C_{ij}=1$ if the $i$-th set contains the element $j$ and zero otherwise.}

\begin{mytheorem}[Computational Complexity of Problem \eqref{pr:min_set}]\label{th:NP_1}
{Problem \eqref{pr:min_set} is NP-hard.}
\end{mytheorem}
\begin{proof}
{We show that Problem \eqref{pr:min_set} for $A$ as described below and with $E=n(2n)^{2n^2+12n +2}-n$ is equivalent to the NP-hard controllability problem introduced in \cite{2013arXiv1304.3071O}.  Therefore, since $E$ can be described in polynomial time, as $log(E)=O(n^3)$, we conclude that Problem \eqref{pr:min_set} is NP-hard.}

{In particular, as in \cite{2013arXiv1304.3071O}, let $n = m+p+1$ and $A = V^{-1} D V$, where $D \equiv \text{diag}(1,2,\ldots,n)$ and\footnote{$V$ is invertible since it is strictly diagonally dominant.}}
\begin{align} 
V= \left[ 
\begin{array}{ccc}
2I_{m\times m} & 0_{m \times p} & e_{m\times 1} \\
C & (m+1)I_{p\times p} & 0_{p \times 1} \\
0_{1 \times m} & 0_{1 \times p} & 1
\end{array}\right].
\end{align}
{It is shown in \cite{2013arXiv1304.3071O} that deciding if $A$ is controllable by a zero-one diagonal matrix $B$ with $r+1$ non-zero entries is NP-hard.}

{Now, observe that all the entries of $V$ are integers either zero or at most $m+1$. Moreover, with respect to the entries of $V^{-1}$, it is shown in \cite{2013arXiv1304.3071O} that:
\begin{itemize}
\item For $i=1,2,\ldots, m$: It has a $1/2$ in the $(i,i)$-th place and a $-1/2$ in the $(i,n)$-th place, and zeros elsewhere.
\item For $i=m+1,m+2,\ldots,m+p$: It has a $1/(m+1)$ in the $(i,i)$-th place, a $-1/(2(m+1))$ in the $(i,j)$-th place where $j \in C_i$ ($C_i$ is the corresponding set of the collection $\mathcal{C}$), and $|C_i|/(2(m+1))$ in the $(i,n)$-th place; every other entry of the $i$-th row is zero.
\item Finally, the last row of $V^{-1}$ is $[0,0,\ldots,0,1]$.
\end{itemize}
Therefore, $2(m+1)V^{-1}$ has all its entries as integers that are either zero or at most $n^2$, in absolute value.}

{Consider the controllability matrix associated with this system, given a zero-one diagonal $B$ that makes it controllable, and denote it as $W_B$.  Then,
\begin{align*}
W_B&=\int_{t_0}^{t_1} \mathrm{e}^{{A}(t-t_0)} {B}{B}^{T} \mathrm{e}^{{A}^{T}(t-t_0)}\,\mathrm{d}{t}\\
&=V^{-1}\int_{t_0}^{t_1} \mathrm{e}^{{D}(t-t_0)} V{B}V^T \mathrm{e}^{{D}^{T}(t-t_0)}\,\mathrm{d}{t}V^{-T}.
\end{align*}
}

{Let $t_1-t_0=ln(n)$.  Then, $(2n)!\int_{0}^{t_1-t_0} \mathrm{e}^{{D}t} V{B}V^T \mathrm{e}^{{D}^{T}t}\,\mathrm{d}{t}$ evaluates to a matrix that has entries of the form $c_0+c_1n+c_2n^2+\ldots+c_nn^n$, where $c_0, c_1, \ldots, c_n$ are non-negative integers and all less than $(2n)! \leq (2n)^{2n}$.  Thereby,
\[
W'_B\equiv4(m+1)^2(2n)!V^{-1}\int_{0}^{t_1-t_0} \mathrm{e}^{{D}t} V{B}V^T \mathrm{e}^{{D}^{T}t}\,\mathrm{d}{t}V^{-T},
\]
has entries of the form $c'_0+c'_1n+c'_2n^2+\ldots+c'_nn^n$, where $c'_0, c'_1, \ldots, c'_n$ are integers and all less than $(2n)^{2(n+3)}$ in absolute value due to the pre and post multiplications by $2(m+1)V^{-1}$ and $2(m+1)V^{-T}$, respectively.}

{We are interested on upper bounding $\log\det(W_B^{-1})$: since for $x>0$, $\log(x)\leq x-1$, $\log\det(W_B^{-1})\leq tr(W_B^{-1})-n$.  In addition, $ tr(W_B^{-1})=4(m+1)^2(2n)!tr({W'_B}^{-1})\leq (2n)^{2(n+1)}tr({W'_B}^{-1})$.  Therefore, we upper bound $tr({W'_B}^{-1})$: Using Crammer's rule to compute ${W'_B}^{-1}$, due to the form of the entries of ${W'_B}$, all of its elements, including the diagonal ones, if they approach infinity, they approach it with at most $n! n^n(2n)^{2n(n+3)}$ $<$ $  (2n)^{2n(n+5)}$ speed, and as a result $tr({W'_B}^{-1})\leq n(2n)^{2n(n+5)}$.  Hence, $tr(W_B^{-1})\leq n(2n)^{2n(n+5)+2(n+1)}=n(2n)^{2n^2+12n +2}$, for any $B$ that makes \eqref{eq:dynamics} controllable.  Thus, if we set $E= n(2n)^{2n^2+12n +2}-n$ (which implies $log(E)=O(n^3)$ so that $E$ can be described polynomially), Problem \eqref{pr:min_set} is equivalent to the controllability problem of \cite{2013arXiv1304.3071O}, which is NP-hard.}
\end{proof}

An immediate consequence of Theorem \ref{th:NP_1} is the following one. 

\begin{mycorollary}[Computational Complexity of Problem \eqref{pr:original}]
{Problem \eqref{pr:original} is NP-hard.}
\end{mycorollary}

\subsection{The Greedy Algorithm used in the Supermodular Minimization Literature is Inefficient for solving Problem~\eqref{pr:original_approx}}\label{appen:badalg}

Consider Algorithm~\ref{alg:r-leaders} which is in accordance with the supermodular minimization literature~\cite{Nemhauser:1988:ICO:42805, citeulike:416650,krause2012submodular}.

\begin{algorithm}[H]
\caption{Greedy algorithm for Problem~\eqref{pr:original_approx}.}\label{alg:r-leaders}
\begin{algorithmic}
\REQUIRE Maximum number of actuators $r$, approximation parameter $\epsilon$, number of steps that the algorithm will run $l$, matrices ${W}_1, {W}_2, \ldots, {W}_n$.
\ENSURE Actuator set $\Delta_l$
\STATE $\Delta_0\leftarrow\emptyset$ , $i\leftarrow 0$
\WHILE {$i < l$} \STATE 
	 $a_i \leftarrow \text{argmax}_{a \in \mathcal{V}\setminus \Delta}\{
	\log\det({W}_{\Delta_i}+\epsilon{I})^{-1}-\log\det({W}_{\Delta_i\cup \{a\}}+\epsilon{I})^{-1}
	\}$\\
    \quad $\Delta_{i+1} \leftarrow \Delta_i \cup \{a_i\}, i\leftarrow i+1	$
	
\ENDWHILE
\end{algorithmic} 
\end{algorithm}

The following is true for its performance. 

\begin{myfact}
Let $v^\star$ denote the value of Problem~\eqref{pr:original_approx}.  Then,   Algorithm~\ref{alg:r-leaders} guarantees that for any positive integer $l$,
\begin{align*}
\log\det({W}_{\Delta_l}+\epsilon{I})^{-1}\leq (1-e^{-l/r})v^\star+n\log(\epsilon^{-1})e^{-l/r}. 
\end{align*}
\end{myfact}
\begin{proof}
It follows from Theorem 9.3, Ch.~III.3.9. of~\cite{Nemhauser:1988:ICO:42805}, since $-\log\det({W}_{\Delta_l}+\epsilon{I})^{-1}+n\log(\epsilon^{-1})$ is a non-negative, non-decreasing, and submodular function with respect to the choice of $\Delta$ (cf.~Proposition~\ref{prop:subm}).  
\end{proof}

Algorithm~\ref{alg:r-leaders} suffers from an error term that is proportional to $n\log(\epsilon^{-1})$.    Moreover, it is possible that Algorithm~\ref{alg:r-leaders} returns an actuator set that does not render~\eqref{eq:dynamics} controllable.  Therefore, Algorithm~\ref{alg:r-leaders} is inefficient for solving Problem~\eqref{pr:original_approx}.


\bibliographystyle{IEEEtran}
\bibliography{newRef}

\end{document}